\documentclass[12pt]{amsart}
\usepackage{amsmath,amssymb}
\usepackage[margin=1in]{geometry} 
\usepackage{todonotes}

\newcommand{\R}{\ensuremath{\mathbb{R}}}

\newtheorem{thm}{Theorem}

\newtheorem{lem}[thm]{Lemma}
\newtheorem{defn}[thm]{Definition}

\DeclareMathOperator{\bag}{bag}

\title[Higher-order chain rules without losing your marbles]
{How to stare at the higher-order $n$-dimensional Chain-rule
without losing your marbles}
\author{Henry O. Jacobs}
\date{\today}

\begin{document}

\maketitle

Let $f,g: \R \to \R$.
The formula for the $n$th derivative of $f\circ g$ is
given by Fa\`a di Bruno's formula
\begin{align}
  \begin{split}
  &\frac{d^n}{dx^n} (f \circ g)(x) = \\
  &\quad
  \sum_{k=1}^{n}
  \left(
    \sum_{
      \substack{
        m_1 + \cdots + m_n = k \\
        m_1 + 2 m_2 + \cdots + n m_n = n
      }
    }
      \frac{n!}{m_1! \cdots m_n !} g^{(k)}(f(x)) 
      \left( \frac{f^{(1)}(x)}{1!} \right)^{m_1}
      \cdots
      \left( \frac{f^{(n)}(x)}{n!} \right)^{m_n}
  \right)
  \end{split}\label{eq:Faa1}
\end{align}
A typical reaction to a first encounter with \eqref{eq:Faa1}
is described in the first page of \cite{Flanders2001} wherein
the author was asked to provide a proof of the above formula:

\begin{quotation}
  This exercise is hardly routine calculus!  [...]
  All those factorials in the denominators, raised to powers yet!
  My four years of Chicago high school mathematics: Algebra, Advanced Algebra, [...],
  calculus hardly prepared me for Fa\`a's formula.
\end{quotation}
Fortunately, \cite{Flanders2001} manages to make sense of the formula; by
justifying the restrictions on the $m$'s and all those factorials.
In fact, over the past two centuries \eqref{eq:Faa1} has been
viewed from a variety of perspectives; such as Bell polynomials,
set partitions, determinant formulas, and so on \cite{Johnson2002}.

However, there is very little written on the multidimensional 
generalization of \eqref{eq:Faa1}.
To provide perspective on the issue, 
we can state the multidimensional higher-order
chain rule at orders $1,2$, and $3$.
Let $\partial_{i}$ denote the partial differentiation operator
along the $i$th coordinate direction,
and let $\partial_{ij} = (\partial_i \circ \partial_j), \partial_{ijk} = (\partial_i \circ \partial_j \circ \partial_k)$.
If $g = (g^1,\dots,g^d) : \R^c \to \R^d$ and $f:\R^d \to \R$, then
\begin{align*}
  \partial_i (f \circ g)(x) &= \sum_{j=1}^{d}\partial_jf(g(x)) 
  (\partial_i g^j(x)) \\
  \partial_{ij}(f \circ g)(x) &= \sum_{k,\ell = 1}^{d} \partial_{k\ell}f(g(x))
  (\partial_i g^k(x) ) (\partial_j g^{\ell}(x))
  + \sum_{k=1}^{d}\partial_k f(g(x)) \partial_{ij} g^k(x)
\end{align*}
and finally,
\begin{align*}
  &\partial_{ijk} (f \circ g)(x) = 
  \sum_{\ell,m,n=1}^{d} \partial_{\ell m n} f(g(x)) 
  (\partial_i g^\ell(x))  (\partial_j g^m(x))
  (\partial_k g^n(x)) \\
  &\quad + \sum_{mn=1}^{d} \partial_{mn}f(g(x)) \left(
  ( \partial_j g^m(x) )(\partial_{ik} g^n(x)) +
  ( \partial_k g^m(x) ) (\partial_{ij} g^n(x)) +
  ( \partial_i g^m(x) ) (\partial_{jk} g^n(x) \right) \\
  &\quad + \sum_{\ell=1}^{d} \partial_\ell f(g(x)) \partial_{ijk}g^\ell(x).
\end{align*}
where $i,j,k = 1,\dots,c$.

It is natural to ask for a generic formula for
$\partial_{i_1 \cdots i_n} (f \circ g)(x)$.
One must consider partitions of the tuple $(i_1,\dots,i_n)$,
and account for equivalences such as
``$\partial_{ij} \equiv \partial_{ji}$''.
Simply printing the formula using a standard multi-index
convention is a formidable task
and we are not able to present it here.
Instead, we refer the reader to the statement and proof in
 \cite{ConstantineSavits1996}.
It is notable that the proof in \cite{ConstantineSavits1996}
consists of six full pages of equations, broken down into four lemmas.
Much effort goes into simplyfying expressions and keeping
the combinatorial complexity under control.
This is in spite of good notation and an efficient proof!

In this paper, we will introduce a version of the multidimensional higher-order
chain rule, in which all of the coefficients are equal to $1$ and the proof
 is relatively simple.  This is done by putting greater emphasis
 on the algebraic structure of the indices.

\section{Multiset indices}
  Perhaps we could place (hide?) the
  combinatoric considerations in the indexing convention itself.
  This entails summing over a smaller set of more sophisticated
  indices.
  One standard multi-indexing convention in $d$ variables
  is to define the multi-indices as tuples of $\{1,\dots,d\}$.
  Given a tuple $(a_1,\dots,a_n)$ we can consider the partial
  differential operator  $\partial_{a_1 \cdots a_n}$.
  However, any two multi-indices which are equivalent modulo
  permutation will generate the same partial differential operator
  due to the equivalence of mixed partials.
  To remedy this, one might be tempted to consider the set $\{ a_1,\dots,a_n\}$
  rather than the tuple $(a_1,\dots,a_n)$, because this ``mods out'' the permutation symmetry.
  However, resorting to sets removes multiplicities.
  There is no way to represent the partial differential operator $\partial_x^2 = \partial_{xx}$
  using just sets because $\{x,x\} = \{x\}$.
  Thus we seek a convention which respects multiplicities and permutation symmetry.
  \begin{defn}
    A \emph{multiset} (or \emph{bag}) is a pair $(A,m)$ where $A$ is a set and 
    $m$ is a map from $A$ to $\mathbb{N} = \{0,1,\dots\}$.
    Given two multisets $(A_1,m_1)$ and $(A_2,m_2)$
    we define the union to be the multiset
    $(A_1,m_1) \cup (A_2,m_2) := (A_1\cup A_2, m_1 + m_2)$,
    where $m_1 + m_2$ is shorthand for the function
    \begin{align*}
    	(m_1 + m_2)(x) = \begin{cases}
		m_1(x) & x \in A_1 \backslash A_2 \\
		m_2(x) & x \in A_2 \backslash A_1 \\
		m_1(x) + m_2(x) & x \in A_1 \cap A_2
		\end{cases}
    \end{align*}
    Given $x_1,\dots,x_n \in S$ we let $[x_1,\dots,x_n]$
    denote the multiset $(\{x_1,\dots,x_n\},m)$ where $m(x)$ denotes the multiplicity
    of $x$ in the sequece $x_1,\dots,x_n$.
    We call $A$ the underlying set of the multiset $(A,m)$,
    and we call $|(A,m)| := \sum_{x \in A}m(x)$ the cardinality
    of $(A,m)$.
  \end{defn}

    Consider the multiset $[1,1,2]$ and a map $f:\mathbb{N} \to \mathbb{R}$.
    We'd like to write the sum $f(1) + f(1) + f(2)$
    as $\sum_{x \in [1,1,2] } f(x)$.
    This motivates the following convention.
    If $(A,m)$ is a multiset and $f:A \to \mathbb{R}$, then
    \begin{align*}
      \sum_{ x \in (A,m) } f(x) := \sum_{x \in A } m(x) f(x),
    \end{align*}
    where the right hand side is a standard summation.
  \begin{defn}
    A multiset index of $d$ variables is a multiset $\alpha$
    whose underlying set is $\{1,\dots,d\}$. 
    We will denote multiset indices by greek letters rather than 
    as pairs of sets and multiplicity functions.
    The set of multiset indices on $\{1,\dots,d\}$ is denoted by $\bag(d)$,
    and the subset of which have cardinality $n$ is denoted by $\bag^n(d)$.
  \end{defn}

  Algebraically, a multiset index is an element of the free commutative
  module generated by the integers $1,\dots,d$.
  Heuristically, a multiset index is nothing but a bag of
  $n$ marbles, which come in colors $1,\dots,d$.
  Given $\alpha \in \bag^n(d)$ and $\beta \in \bag^m(d)$,
  the union $\alpha \cup \beta \in \bag^{n+m}(d)$ is the bag of $n+m$ marbles obtained
  by combing the bags $\alpha$ and $\beta$
  \cite{Blizard1989}.
  It is notable that the multi-indices used in \cite{ConstantineSavits1996}
  are equivalent to multiset indices.
  Specifically, they used the multiplicity function itself as an index.
  However, they did not use any of the multi-set structures which we are
  about to invoke here. Viewing the indices as 
  ``bags of stuff'' is particularly powerful, and
  we will find that the structure induced by this perspective greatly simplifies the
  derivation (and expression) of the chain rule.

  A labeling of a multiset $(A,m)$ of cardinality $n \in \mathbb{N}$
  is a map $a:\{1,\dots,n\} \to S$
  such that the cardinality of the set $a^{-1}(x) = \{ j : a(j) = x \}$
  is equal $m(x)$.
  Equivalently, a labelling of $(A,m)$ is just a tuple 
  $(a_1,\dots,a_{n})$ such that $[a_1,\dots,a_{n}] = (A,m)$.

  For a multiset index $\alpha$ we let $\partial_{\alpha}$
  denote the partial differential operator
  $\partial_{a_1 \cdots a_n}$ for an arbitrary labeling $(a_1,\dots,a_n)$
  of $\alpha$.
  Note that the chosen labelling of $\alpha$
  is immaterial due to the equivalence of mixed partials.
  Given this convention we observe
  $\partial_{\alpha} \partial_{\beta} = \partial_{\beta} \partial_{\alpha} = \partial_{ \alpha \cup \beta}$
  for any two multiset indices $\alpha$ and $\beta$.

  An important concept which we will use is the notion of a partition.
  For any set $S$ and any $k \in \mathbb{N}$,
  a $k$th order set-partition is a set of non-empty disjoint sets
  $\{ S_1,\dots,S_k\}$
  such that $S_1\cup \cdots \cup S_k = S$.
  We denote the set of $k$th order set-partitions of $S$ by $\Pi(S,k)$.
  We now generalize this notion to the case of multisets.

\begin{defn}
  Let $(A,m)$ be a multiset.  A $k$th order multiset-partition of $(A,m)$
  is a multiset of multisets $[(A_1,m_1),\dots,(A_k,m_k)]$
  such that $(A_1,m_1) \cup \dots \cup (A_k,m_k) = (A,m)$.
\end{defn}

If $(A,m)$ is a multi-set with cardinality $n$,
we can generate a $k$th order multiset partition
by considering a labeling $a:\{1,\dots,n\} \to A$
and considering a $k$th order set-partition of $\{1,\dots,n\}$.
In particular, if $\{S_1,\dots,S_k\}$ is a
partition of $\{1,\dots,n\}$
we can define the multiset $(A_i,m_i)$ where $A_i = a(S_i)$
and the multiplicity of $x \in A_i$ is given by the number of
elements of $S_i$ which map to $x$ under $a$.
Explicitly, $m_i( x ) = | \{ k \in S_i : a(k) = x \} |$
for each $x \in A_i$.
It follows that
 $[(A_1,m_1), \dots, (A_k,m_k)]$ is
a $k$th order multiset-partition of $(A,m)$.
Note that two distinct set-partitions of $\{1,\dots,n\}$
can generate the same multiset-partition.
Thus the space of multiset partitions of a multiset
generated in this way, has multiplicity.

\begin{defn}
  We let $\Pi((A,m),k)$ denote the multiset of $k$th order
  multiset-partitions of $(A,m)$.
  The multiplicity of 
  a multiset partition $[(A_1,m_1),\dots,(A_1,m_1)] \in \Pi(\alpha,k)$
  is defined as the number of partitions of $\{1,\dots,n\}$
  which generate it.
\end{defn}

Firstly, note that $\Pi((A,m),k)$ is independent of any labelling
we choose to generate it, as all labellings are equivalent up to permutations.
Secondly, note that the cardinality of the multiset $\Pi((A,m),k)$
is identical to the cardinality of the set $\Pi(\{1,\dots,n\},k)$
where $n = | (A,m)|$.
However, there are generally fewer \emph{distinct} multiset partitions
because we allow them to be repeated.

As an example consider the multiset index $[1,1,2]$.
The set $\{1,2,3\}$ has three distinct $2$nd order set-partitions:
$\{ \{1\},\{2,3\} \} , \{ \{2\},\{1,3\} \}$, and $\{ \{3\},\{1,2\} \}$.
Thus $| \Pi( [1,1,2] , 2) | = |\Pi( \{1,2,3\} , 2)| = 3$.
We find that $\Pi([1,1,2],2)$ is the multiset with partitions 
$[ [1],[2,1]]$ , $[ [2],[1,1] ]$, and $[ [1],[1,2] ]$.
Note that the first and the third multiset-partitions correspond to the
same multiset.  Thus $\Pi([1,1,2],2)$ has only $2$ distinct elements
but a cardinality of $3$.

\begin{thm} \label{thm:main}
Let $f:\R^c \to \R$ and $g:\R^d \to \R^c$.  Then
\begin{align*}
  \partial_\alpha( f \circ g)(x) = \sum_{n=1}^{|\alpha|} \left(
  \sum_{ 
    \substack{
      b_1 ,\dots,b_n\in \{1,\dots,c\} \\
      [\alpha_1,\dots,\alpha_n] \in \Pi(\alpha,n)
     }
    }
    \partial_{b_1\cdots b_n} f|_{g(x)}
    \prod_{k=1}^{n} \partial_{\alpha_k} g^{b_k}(x)
  \right),
\end{align*}
  for any $\alpha \in \bag(d)$.
\end{thm}

Before we prove the theorem we consider the following lemma.
\begin{lem}
  Let $\alpha \in \bag(d)$ and $[a_0] \in \bag^1(d)$.
  Any $(n+1)^{\rm th}$ order multiset partition of $[a_0] \cup \alpha$
  is either of the form $[a_0,\alpha_1,\dots,\alpha_n]$
  for some $[\alpha_1,\dots,\alpha_n] \in \Pi(\alpha,n)$,
  or of the form $[ \alpha_1, \dots, a_0 \cup \alpha_{k} , \dots, \alpha_{n+1}]$
  for some $[\alpha_1,\dots,\alpha_{n+1}] \in \Pi(\alpha,n+1)$.
  Moreover, the given multiset which includes $a_0$ has the 
  same multiplicity as $[\alpha_1,\dots,\alpha_n] \in \Pi(\alpha,n)$
  or $[\alpha_1,\dots,\alpha_{n+1}] \in \Pi(\alpha,n+1)$.
\end{lem}
\begin{proof}
  Let $[\delta_0,\dots, \delta_{n}] \in \Pi( [a_0] \cup \alpha,n+1)$.
  Then $a_0$ is contained in some $\delta_{k}$.
  We consider to complementary but disjoint scenarios.
  Either $\delta_k = [a_0]$ for some $k$, or not.

  If $\delta_{k} = [a_0]$
  then the remaining delta's must partition $\alpha$.
  In other words,
  $[\delta_0,\dots,\delta_{n}] = [ [a_0], \alpha_1,\dots,\alpha_n]$
  for some $[\alpha_1,\dots,\alpha_n] \in \Pi(\alpha,n)$.
  Let $p = |\alpha|$.
  If $a:\{1,\dots,p\} \to \{1,\dots,d\}$ is a labelling of $\alpha$
  we can see that any set partition $\{S_1,\dots,S_n\} \in \Pi(\{1,\dots,p\},n)$
  which generates $[\alpha_1,\dots,\alpha_n]$
  can be put in one-to-one correspondence
  with the set partition $\{ \{0\} , S_1,\dots,S_n\} \in \Pi( \{0,\dots,p\},n)$.
  The later set partition generates the multiset partition
  $[ [a_0],\alpha_1,\dots,\alpha_n] = [\delta_0,\dots,\delta_n]$.
  Thus the multiplicity of $[\delta_0,\dots,\delta_n] \in \Pi(\alpha,n+1)$
  is identical to that of $[\alpha_1,\dots,\alpha_n] \in \Pi(\alpha,n)$.

  Otherwise, $a_0$ is contained in a $\delta_k$ of the form
  $a_0 \cup \alpha_k$ for some multiset index $\alpha_k$.
  Again, $\alpha_k$ along with the remaining $\delta$'s 
  must partition $\alpha$.
  If we set $\alpha_j = \delta_j$ for $j\neq k$ we see that 
  $[\alpha_0,\dots,\alpha_n] \in \Pi(\alpha,n+1)$.
  The multiplicity of $[\alpha_0,\dots,\alpha_n]$ 
  is identical to that of $[\delta_0,\dots,\delta_n]$
  by virture of the same argument used in the previous case.
\end{proof}

We now proceed to prove the main theorem.

\begin{proof}
  We prove it inductively.  It holds by inspection at order $1$.
  Assume it holds for some higher order multiset index $\alpha$,
  and let $[a_0] \in \bag^1(d)$.
  By the product formula and chain rule, we find
  \begin{align*}
    &\partial_{a_0}\partial_\alpha( f \circ g)(x) = 
    \partial_{[a_0] \cup \alpha}( f \circ g)(x) \\
    &\quad =\sum_{n=1}^{|\alpha|}
    \sum_{ 
      \substack{
        b_1,\dots,b_n \in \{1,\dots,c\} \\
        [\alpha_1,\dots,\alpha_n] \in \Pi(\alpha,n) 
        }
      }
      \left( \sum_{b_0=1}^d
        \partial_{b_0 b_1 \cdots b_n} f(g(x))
        \partial_{a_0} g^{b_0}(x)
        \partial_{\alpha_1} g^{b_1}(x) \cdots
        \partial_{\alpha_n} g^{b_n}(x)
      \right)\\
      &\qquad + \partial_{b_1 \cdots b_n} f(g(x))
      \partial_{a_0 \cup \alpha_1} g^{b_1}(x) \cdots
      \partial_{\alpha_n} g^{b_n}(x) \\
      &\qquad + \partial_{b_1 \cdots b_n} f(g(x))
      \partial_{\alpha_1} g^{b_1}(x) 
      \partial_{a_0 \cup\alpha_2} g^{b_2}(x)  \cdots
      \partial_{\alpha_n} g^{b_n}(x) \\
      &\qquad \vdots \\
      &\qquad + \partial_{b_1 \cdots b_n} f(g(x))
      \partial_{\alpha_1} g^{b_1}(x) \cdots
      \partial_{a_0\cup\alpha_n} g^{b_n}(x) . 
  \end{align*}
Let us now collect all coefficients of $\partial_{b_0 \cdots b_n} f(g(x))$.
We observe that this coefficient is
\begin{align*}
&   \left( \sum_{ [\alpha_1,\dots,\alpha_n] \in \Pi( \alpha , n) }
   \partial_{a_0}g^{b_0}(x) \partial_{\alpha_1} g^{b_1}(x) \cdots \partial_{\alpha_n}g^{b_n}(x) \right) \\
& +
 \left( \sum_{k=0}^{n} \sum_{ [\alpha_0, \dots,\alpha_n] \in \Pi( \alpha,n+1) }
  \partial_{\alpha_0}g^{b_0} \cdots \partial_{a_0 \cup \alpha_k} g^{b_k} \cdots \partial_{\alpha_n} g^{b_n}\right).
\end{align*}
By the lemma, one could write this more succinctly as
\begin{align*}
  \sum_{[[\alpha_0],\dots,\alpha_n] \in \Pi( a_0 \cup \alpha , n+1) }
  \partial_{\alpha_0}g^{b_0}(x)
  \cdots
  \partial_{\alpha_n}g^{b_n}(x).
\end{align*}
We substitute the coefficent of
$\partial_{b_0 \cdots b_n} f(g(x))$ into our original formula to arrive at
\begin{align*}
  \partial_{[a_0] \cup \alpha}( f \circ g)(x) &= \sum_{n=0}^{|\alpha |} \sum_{
    \substack{
      b_0,\dots,b_n \in \{ 1,\dots,c\}\\
      [\alpha_0,\dots,\alpha_n] \in \Pi ([a_0] \cup \alpha ,n+1)      
      }
    }\partial_{b_0 \cdots b_n}f(g(x)) \partial_{\alpha_0}g^{b_0}(x) \cdots \partial_{\alpha_n}g^{b_n}(x) \\
  &=\sum_{n=1}^{| [a_0] \cup \alpha |} \sum_{
    \substack{
      b_1,\dots,b_n \in \{1,\dots,c\}\\
      [\alpha_1,\dots,\alpha_n] \in \Pi( [a_0] \cup\alpha,n)      
      }
    }\partial_{b_1 \cdots b_n}f(g(x)) \partial_{\alpha_1}g^{b_1}(x) \cdots \partial_{\alpha_n}g^{b_n}(x).
\end{align*}
Thus we have proven the formula for an arbitrary multi-set index of cardinality $|\alpha| + 1$.
\end{proof}

One critique that can be lodged is that Theorem \ref{thm:main} invokes two
indexing convections.  It uses standard multi-indices, via the $b$'s,
and it uses multiset indices, via $\alpha$ and its multiset partitions.
One quick fix for this is to define a new notation.
Given $\beta \in \bag^n(d)$ and $[\alpha_1,\dots,\alpha_n] \in \Pi(\alpha,n)$,
\begin{align*}
	\partial_{[\alpha_1,\dots,\alpha_n]} g^\beta(x) := \sum_{[b_1,\dots,b_n] = \beta} \prod_{k=1}^{n} \partial_{\alpha_k} g^{b_k}(x)
\end{align*}
where the outer sum is over all labelings of $\beta$.
This allows us to replace the sum over the $b$'s in Theorem \ref{thm:main}
as a sum over multiset indices.
\begin{thm}
	If $f:\R^c \to \R$, $g :\R^d \to \R^c$, then
	\begin{align*}
		\partial_\alpha( f \circ g) (x) = \sum_{n=1}^{|\alpha|} \sum_{
			\substack{
				\beta \in \bag^n(c) \\
				[\alpha_1,\dots,\alpha_n] \in \Pi(\alpha,n)}
				 } \partial_\beta f|_{g(x)} \partial_{[\alpha_1,\dots,\alpha_n]} g^\beta(x)
	\end{align*}
	for any $\alpha \in \bag(d)$.
\end{thm}
In the case where $f,g:\R \to \R$ this version of the higher-order chain
rule is identical to the set partition version
of the Fa\`a di Bruno formula shown on page 3 of \cite{Johnson2002}.

\section{Conclusion}
While the $n$-dimensional higher-order chain rule may be difficult
to tackle when using standard multi-indices, it appears
relatively simple when using multiset indices.
This is not to say that multiset indexing is a superior indexing convention.
The multiset of multiset partitions $\Pi(\alpha,n)$
can be irritating to enumerate.
It is quite conceivable that one would prefer to 
enumerate over all tuples of integers and then divide by the 
number of repeated terms.
However, the combinatorial coefficients can be difficult to compute and
interpret.
Therefore, it is useful to have an alternative which disposes of them
in place of more tactile objects.

\section{Acknowledgements}
I owe a special thanks to Jaap Eldering for meticulously checking 
virtually every nook and cranny of this article.
This research is supported by European Research Council Advanced Grant 267382.

\bibliographystyle{amsalpha}

\providecommand{\bysame}{\leavevmode\hbox to3em{\hrulefill}\thinspace}
\providecommand{\MR}{\relax\ifhmode\unskip\space\fi MR }
\providecommand{\MRhref}[2]{%
  \href{http://www.ams.org/mathscinet-getitem?mr=#1}{#2}
}
\providecommand{\href}[2]{#2}

\end{document}